\documentclass[11pt]{amsart}
\usepackage{a4wide,mathtools}
\usepackage{color,amssymb}
\usepackage[all]{xy}

\newcommand{\defeq}{\coloneqq}
\newcommand{\IR}{\mathbb R}
\newcommand{\IN}{\mathbb N}
\newcommand{\w}{\omega}

\newcommand{\F}{\mathcal F}
\newcommand{\cl}{\mathrm{cl}}
\newcommand{\Fr}{\mathfrak{Fr}}
\newcommand{\Tau}{\mathcal T}

\newtheorem{theorem}{Theorem}[section]
\newtheorem{proposition}[theorem]{Proposition}
\newtheorem{corollary}[theorem]{Corollary}
\newtheorem{lemma}[theorem]{Lemma}
\newtheorem{claim}[theorem]{Claim}
\newtheorem{question}[theorem]{Question}
\newtheorem{example}[theorem]{Example}

\title[Strongly rigid countable spaces]{Examples of strongly rigid countable\\ (semi)Hausdorff spaces}
\author{Taras Banakh and Yaryna Stelmakh}
\address{Ivan Franko National University of Lviv, Ukraine}
\email{t.o.banakh@gmail.com, yarynziya@ukr.net}

\begin{document}
\begin{abstract} A topological space $X$ is {\em strongly rigid} if each non-constant continuous map $f:X\to X$ is the identity map of $X$.
A Hausdorff topological space $X$ is called {\em Brown} if for any nonempty open sets $U,V\subseteq X$ the intersection $\overline U\cap\overline V$ is infinite. We prove that every second-countable Brown Hausdorff space $X$ admits a stronger topology $\Tau'$ such that $X'=(X,\Tau')$ is a strongly rigid  Brown space. This construction yields an example of a countable anticompact Hausdorff space $X$ which is strongly rigid,  which answers two problems  posed at {\tt MathOverflow}.
By the same method we construct a strongly rigid semi-Hausdorff $k_2$-metrizable space containing a non-closed compact subset, which answers two other problems posed at {\tt MathOverflow}.
\end{abstract}
\maketitle

\section{Introduction}

A topological space $X$ is called
\begin{itemize}
\item {\em rigid} if every homeomorphism $f:X\to X$ coincides with the identity map of $X$;
\item {\em strongly rigid} if every non-identity continuous map $f:X\to X$ is constant.
\end{itemize}

\begin{proposition} Every strongly rigid space $X$ is connected.
\end{proposition} 

\begin{proof} Assuming that $X$ is disconnected we can write $X$ as the union $X=U\cup V$ of two nonempty disjoint  open sets $U,V$. Choose any points $u\in U$ and $v\in V$ and consider the continuous map $f:X\to \{u,v\}$ such that $f^{-1}(u)=V$ and $f^{-1}(v)=U$. This map witness that the space $X$ is not strongly homogeneous.
\end{proof}

 In \cite{Cook} Cook constructed an example of a strongly rigid compact metrizable space. The Cook space being metrizable and connected, has cardinality of continuum.

In \cite{MO} the user {\tt QuinnLesquimau} asked whether there exists a {\em countable} strongly rigid Hausdorff space. The same question was posed earlier by Kannan and Rajagopalan \cite[p.128]{KR} who constructed in \cite{KR} a countable strongly rigid space under a suitable set-theoretic assumption. In this paper we answer this question in affirmative suggesting a general construction of strongly rigid countable Hausdorff spaces.\footnote{After writing this paper we have learned that another answer to the question of {\tt QuinnLesquimau} was given by K.P. Hart who observed in  ({\tt https://mathoverflow.net/a/434266/61536}) that after a suitable modification the construction of Kannan and Rajagopalan \cite{KR} does produce a strongly rigid countable Hausdorff space in ZFC.}

Let us observe that a strongly rigid countable topological space necessarily has poor separation properties. In particular, it cannot be functionally Hausdorff.

Let us recall that a topological space $X$ is
\begin{itemize}
\item {\em functionally Hausdorff} if for any distinct points $x,y\in X$ there exists a continuous function $f:X\to \IR$ such that $f(x)\ne f(y)$;
\item {\em semi-Hausdorff} if for any distinct points $x,y\in X$ there exists a regular open set $U\subseteq X$ such that $x\in U$ and $y\notin U$;
\item {\em Urysohn} if for any distinct points $x,y\in X$ there exist open sets $U,V\subseteq X$ such that $x\in U$, $y\in V$ and $\overline U\cap\overline V=\emptyset$;
\item {\em Brown} if  for every nonempty open sets $U,V\subseteq X$ the intersection $\overline U\cap\overline V\ne\emptyset$ is infinite;
\item {\em anticompact} if every compact subspace of $X$ is finite.
\end{itemize}
We recall that a subset of a topological space $X$ is {\em regular open} if it  coincides with the interior of its closure in $X$. Anticompact spaces were introduced and studied by Bankston \cite{Bankston}. All nonempty Brown spaces are infinite and are Brown in the sense of Clark, Lebowitz-Lockard, Polard \cite{CLP} who defined a topological space $X$ to be Brown if $\overline U\cap\overline V\ne\emptyset$ for every nonempty open sets $U,V\subseteq X$. For nonempty Hausdorff spaces containing more than one point our definition of a Brown space agrees with the definition of Clark, Lebowitz-Lockard, Polard \cite{CLP}. 

\begin{proposition} A functionally Hausdorff topological space $X$ of cardinality $1<|X|<\mathfrak c$ is disconnected and hence not strongly rigid.
\end{proposition}

\begin{proof} Choose any distinct points $a,b\in X$ and find a continuous map $f:X\to\IR$ such that $f(a)< f(b)$. Since $|f[X]|\le |X|<\mathfrak c$, there exists a real number $c\in [f(a),f(b)]\setminus f[X]$. Then the disjoint nonempty open sets $O_a\defeq\{x\in X:f(x)<c\}$ and $O_b\defeq\{x\in X:f(x)>c\}$ witness that the space $X$ is disconnected. The continuous function $g:X\to\{a,b\}$ with $g^{-1}(a)=O_b$ and $g^{-1}(b)=O_a$ witnesses that $X$ is not strongly rigid.
\end{proof}

In this paper we shall find many examples of strongly rigid countable Hausdorff spaces among countable Brown spaces (which are not functionally Hausdorff and even are not Urysohn).

The first example of a Brown countable Hausdorff space was constructed by Brown \cite{Brown} and popularized by Golomb \cite{Golomb}, \cite{Golomb61}. The Brown--Golomb space is the space $\IN$ of positive integers, endowed with the topology generated by the base $a+b\IN_0\defeq\{a+bn:n\ge 0\}$ where $a,b\in\IN$ are coprime numbers. The Brown--Golomb space is rigid \cite{BST} but not strongly rigid, see \cite[4.8]{BMT}.

Another classical example of a countable Brown space is the Bing plane \cite{Bing}. By \cite{BBHS}, the Bing plane is topologically homogeneous, so not rigid. Some other examples of Brown spaces can be found in  \cite{BYT}, \cite{CLP},  \cite{Kirch}, \cite{Szcz}.

Brown spaces are connected and moreover, vertex-connected. We define a topological space $X$ to be {\em vertex-connected} if $X$ is connected and for every $x\in X$ the subspace $X\setminus\{x\}$ of $X$ is connected. Therefore, a connected space $X$ is vertex-connected if it contains no cut points. We recall \cite[6.1]{Nad} that a point $x$ of a connected space $X$ is a {\em cut point} of $X$ if the subspace $X\setminus\{x\}$ is disconnected.

A topology $\Tau'$ on a topological space $(X,\Tau)$ is called an {\em open refinement} of the topology $\Tau$ if for every open sets $V\in\Tau$, $V'\in\Tau'$ and point $x\in V\cap V'$ there exists an open set $U\in\Tau$ such that $U\subseteq V\cap V'$ and $\{x\}\cup U\in\Tau'$. This definition implies that $\Tau\subseteq\Tau'$.

\begin{theorem}\label{t:main} Every second-countable Brown Hausdorff space $X$ has an open refinement $\Tau'$ such that the topological space $X'\defeq(X,\Tau')$ is Brown, anticompact, and strongly rigid.  Moreover, every continuous map $f:C\to X'$ from a vertex-connected subspace $C$ of $X'$ to $X'$ is either constant or else $f$  is the identity embedding of $C$ into $X'$.
\end{theorem}

\begin{corollary}\label{c:main} There exists a  strongly rigid anticompact countable Hausdorff space.
\end{corollary}

Corollary~\ref{c:main} answers the {\tt MathOverflow} question \cite{PatrickR} on the existence of an anticompact connected countable Hausdorff space.

A subset $A$ of a topological space $X$ is called a {\em retract} of $X$ if there exists a continuous map $f:X\to A$ such that $f(a)=a$ for all $a\in A$. It is well-known (and easy to see) that every retract in a Hausdorff space is closed. Every retract of a strongly rigid $T_1$-space $X$ is also closed, being a singleton or the whole space $X$. In \cite{JS} Jianing Song posed the problem of the existence of a topological space $X$ such that every retract is closed in $X$ but $X$ contains a non-closed compact subspace. He also asked whether such a space can be a {\em US}-space.

We recall that a topological space $X$ is called a {\em US-space} 
if every convergent sequence in $X$ has a unique limit in $X$. The class of {\em US}-spaces includes all $k$-metrizable and all $k$-discrete $T_1$-spaces. Following \cite{BBK}, we define a topological space $X$ to be {\em $k$-metrizable} (resp. {\em $k$-discrete}) if there exists a bijective continuous proper map $f:M\to X$ from a metrizable (resp. discrete) space $M$ onto $X$. A continuous function $f:X\to Y$ between topological spaces is called {\em proper} if for every compact set $K\subseteq Y$ the preimage $f^{-1}[K]$ is compact. 

A topological space $X$ is called {\em $k_2$-metrizable} (resp. {\em $k_2$-discrete}) if there exists a bijective continuous proper map $f:M\to X$ defined on a metrizable (resp. discrete) space $M$ such that for any continuous map $g:K\to X$ from a compact Hausdorff space $K$, the function $f^{-1}\circ g:K\to M$ is continuous. 

It is well-known that a topological space $X$ is Hausdorff if and only if the diagonal $\Delta_X\defeq\{(x,y)\in X\times X:x=y\}$ is closed in the square $X\times X$. A topological space $X$ is called {\em $k_\Delta$-Hausdorff} if for every continuous map $f:K\to X\times X$ from a compact Hausdorff space $K$ the preimage $f^{-1}[\Delta_X]$  is closed in $K$.

For every $T_1$ topological space the following implications hold:
$$
\xymatrix{
&\mbox{$k_2$-discrete}\ar@{<=>}[d]\ar@{=>}[r]&\mbox{$k_2$-metrizable}\ar@{=>}[r]\ar@{=>}[d]&\mbox{$k_\Delta$-Hausdorff}\ar@{=>}[d]\\
\mbox{anticompact}\ar@{<=>}[r]&\mbox{$k$-discrete}\ar@{=>}[r]&\mbox{$k$-metrizable}\ar@{=>}[r]&\mbox{a {\em US}-space}.
}
$$
Nontrivial implications of this diagram are proved in Section~\ref{s:US}.





The following theorem answers affirmatively the questions of Jianing Song and also the question \cite{Anon} of {\tt Anon} on the existence of a $k_\Delta$-Hausdorff space containing a non-closed compact Hausdorff subspace.

\begin{theorem}\label{t:main2} There exists a strongly rigid countable $k_2$-metrizable semi-Hausdorff Brown space $X$ containing a non-closed compact metrizable subset.
\end{theorem}

\section{Preliminaries}

We denote by $\w$ the set of finite ordinals and by $\IN\defeq\w\setminus\{0\}$ the set of positive integers. For a subset $F\subseteq \IN$, let 
$$
\begin{aligned}
&3F+1\defeq\{3n+1:n\in F\},&&3F-1\defeq\{3n-1:n\in F\},\\
&3F\pm 1\defeq(3F+1)\cup (3F-1),&&\mbox{and }\;\;3F\defeq\{3n:n\in F\}.
\end{aligned}
$$
For a function $f:X\to Y$ and a subset $A\subseteq X$ let $f[A]\defeq\{f(a):a\in A\}$ be the image of $A$ under the function $f$.

We recall that a topological space $X$ is called a {\em $T_1$-space} if every finite subset of $X$ is closed in $X$.

\begin{lemma}\label{l:vertex} Every Brown $T_1$-space $X$ is vertex-connected; moreover, for any finite set $F\subseteq X$ the space $X\setminus F$ is connected.
\end{lemma}

\begin{proof} Assuming that for some finite subset $F\subseteq X$, the space $X\setminus F$ is disconnected, we can find two disjoint nonempty open sets $U_0,U_1$ in $X$ such that $X\setminus F=U_0\cup U_1$. Since $X$ is Brown, the intersection $\overline U_0\cap \overline U_1$ is infinite and hence contains a point $x\in\overline U_0\cap\overline U_1\setminus F$. Then $x\in X\setminus F=U_0\cup U_1$ and hence $x\in U_i$ for some $i\in\{0,1\}$. Then $U_i$ is a neighborhood of $x$, which is disjoint with the set $U_{1-i}$, which contradicts $x\in\overline U_{1-i}$. This contradiction shows that the space $X\setminus F$ is connected and $X$ is vertex-connected.
\end{proof}

\begin{lemma}\label{l:H} Let $(x_n)_{n\in\w}$ be a sequence of pairwise distinct points of a Hausdorff topological space $X$. If the sequence $(x_n)_{n\in\w}$ converges to some point $x\in X\setminus\{x_n:n\in\w\}$, then there exists a sequence $(U_n)_{n\in\w}$ of pairwise disjoint open sets in $X$ such that $x_n\in U_n$ for all $n\in\w$.
\end{lemma}

\begin{proof} The convergence of the sequence $(x_n)_{n\in\w}$ to $x$ implies that for every $n\in\w$ the subset $K_n\defeq\{x\}\cup\{x_k:k\ge n\}$ of $X$ is compact. By the compactness of the set $K_1\not\ni x_0$ and the Hausdorff property of $X$, there exist disjoint open sets $U_0,W_0$ in $X$ such that $x_0\in U_0$ and $K_1\subseteq W_0$. Inductively for every $n\in\IN$ we can choose two disjoint open sets $U_n,W_n$ in $X$ such that $x_n\in U_n\subseteq W_{n-1}$ and $K_{n+1}\subseteq W_n\subseteq W_{n-1}$. Then $(U_n)_{n\in\w}$ is the required sequence of pairwise disjoint sets with $x_n\in U_n$ for all $n\in\w$.\end{proof}

The proof of Theorem~\ref{t:main} exploits a non-trivial result of Kunen \cite{Kunen} on Rudin--Keisler incomparable ultrafilters. We recall \cite{vM} that two ultrafilters $\mathcal U,\mathcal V$ on a set $X$ are called {\em Rudin--Keisler incomparable} if $\beta f(\mathcal U)\ne \mathcal V$ and $\beta f(\mathcal V)\ne \mathcal U$ for every function $f:X\to X$. Here $\beta f$ is the {\em Stone-\v Cech extension} of $f$ assigning to each ultrafilter $\mathcal U$ on $X$ the ultrafilter $\beta f(\mathcal U)\defeq\{S\subseteq X:f^{-1}[S]\in \mathcal U\}$. 
The following important result was proved by Kunen in \cite{Kunen}.

\begin{lemma}\label{l:Kunen} There are continuum many pairwise Rudin--Keisler incomparable ultrafilters on $\IN$.
\end{lemma}

\section{Proof of Theorem~\ref{t:main}}

Fix any Brown second-countable Hausdorff space $(X,\Tau)$. Let $\{U_n\}_{n\in\w}\subseteq\Tau\setminus\{\emptyset\}$ be a countable base of the topology $\Tau$. Being connected and Hausdorff, the Brown space $X$ contains no isolated points.

For every $n\in\IN$ choose a point $\lambda_n\in \overline U_k\cap\overline U_m\setminus\{\lambda_i:i<n\}$ where $n=2^k(2m+1)$. The point $\lambda_n$ exists since $X$ is Brown and hence the intersection $\overline U_k\cap\overline U_m$ is infinite.

For every $x\in X$ and $n\in\IN$, let $O_n(x)\defeq\bigcap\{U_k:k\le n,\;x\in U_k\}$. It is easy to see that $\{O_n(x)\}_{n\in\IN}$ is a neighborhood base of the topology $\Tau$ at $x$.

For every $x\in X$, choose a function $s_x:\IN\to X\setminus\{x\}$ such that the following conditions are satisfied:
\begin{itemize}
\item[(i)] $s_x(n)\in O_n(x)\setminus(\{x\}\cup\{s_x(i):1\le i<n\})$ for every $n\in\IN$;
\item[(ii)] if $x=\lambda_{2^k(2m+1)}$ for some $k,m\in\w$, then $s_x[3\IN+1]\subseteq U_k$ and $s_x[3\IN-1]\subseteq U_m$.
\end{itemize}

By Lemma~\ref{l:H}, for every $x\in X$ there exists a function $S_x:\IN\to \Tau$ such that
\begin{itemize}
\item[(iii)] $s_x(n)\in S_x(n)\subseteq O_n(x)\setminus\{x\}$ for every $n\in\IN$;
\item[(iv)] $S_x(n)\cap S_x(m)=\emptyset$ for every distinct numbers $n,m\in\IN$;
\item[(v)] if $x=\lambda_{2^k(2m+1)}$ for some $k,m\in\w$, then $\bigcup\limits_{n\in\IN}S_x(3n+1)\subseteq U_k$ and $\bigcup\limits_{n\in\IN}S_x(3n-1)\subseteq U_m$.
\end{itemize}

For any $x\in X$ and a subset $F\subseteq \IN$ let $S_x[F]\defeq\bigcup_{n\in F}S_x(n)$. 

Let $r_x:X\to\w$ be the unique function such that $$r_x^{-1}(0)=X\setminus S_x[\IN]\quad\mbox{and}\quad r_x^{-1}(n)=S_x[\{3n-1,3n,3n+1\}]$$ for every $n\in\IN$.

By \cite[1.5.1]{Eng}, the second-countable Hausdorff space $X$ has cardinality $|X|\le\mathfrak c$. By Lemma~\ref{l:Kunen}, there exists a family of pairwise Rudin--Keisler incomparable free ultrafilters $(\F_x)_{x\in X}$ on $\IN$.

Let $\Tau'$ be the topology on $X$ consisting of the sets $W\subseteq X$ such that for every $x\in W$ there exists a set $F\in \F_x$ such that $S_x[3F\pm1]\subseteq W$.

\begin{claim}\label{cl:base} For every $y\in X$, the family
$$\mathcal B_y\defeq\big\{\{y\}\cup S_y[3E\pm 1]:E\in\F_y\big\}$$is a neighborhood base of the topology $\Tau'$ at $y$.
\end{claim}

\begin{proof}Fix any $y\in X$ and $E \in \F_y$. To show that $W\defeq \{y\}\cup S_y[3E\pm 1]$ is $\Tau'$-open  we need to check that for every $x\in W$ there exists a set $F\in \F_x$ such that $S_x[3F\pm1]\subseteq W$. 

If $x= y$ let $F\defeq E$ and observe that $S_x[3E\pm 1]\subseteq \{y\}\cup S_y[3E\pm1]$.

If $x\neq y$ then $x\in S_y(n)$ for some  $n\in 3 E\pm1$. Since $S_y(n)\in \Tau$, there exists $m\in \IN$ such that $O_m(x)\subseteq S_y(n)$. Since $\F_x$ is a  free filter, the set  $F\defeq\{k\in \IN:3k-1\ge m\}$ belongs to $\F_x$. The condition (iii) ensures that $S_x[3F\pm1]\subseteq O_m(x)\subseteq S_y(n)\subseteq W$ and hence $W\in\Tau'$. 
 
 Therefore, the family $\mathcal B_y$ consists of $\Tau'$-open sets. The definition of the topology $\Tau'$ ensures that $\mathcal B_y$ is a neighborhood base of the topology $\Tau'$ at $y$.
\end{proof}

In the following five claims we shall prove that the topology $\Tau'$ has the required properties.

\begin{claim}  $\Tau'$ is an open refinement of the topology $\Tau$ and hence $\Tau\subseteq\Tau'$.
\end{claim}

\begin{proof} Fix any open sets $W\in\Tau$, $W'\in\Tau'$, and a point $x\in W\cap W'$. Find $n\in\IN$ such that $O_n(x)\subseteq W$. By the definition of the topology $\Tau'$, there exists a set $F\in\F_x$ such that $S_x[3F\pm 1]\subseteq W'$. Replacing the set $F$ by its subset $\{k\in F:3k-1\ge n\}\in\F_x$, we can assume that $S_x[3F\pm1]\subseteq O_n(x)\cap W'\subseteq W\cap W'$. It is clear that $S_x[3F\pm1]\in\Tau$. By Claim~\ref{cl:base}, $\{x\}\cup S_x[3F\pm1]\in\Tau'$.
\end{proof}

\begin{claim}\label{cl:Brown} $X'\defeq(X,\Tau')$ is a Brown Hausdorff space.
\end{claim}

\begin{proof} Since $\Tau\subseteq\Tau'$, the topology $\Tau'$ is Hausdorff. Assuming that the space $X'$ is not Brown, we can find two nonempty open sets $U',V'\in\Tau'$ such that the intersection $I\defeq \cl_{X'}(U')\cap\cl_{X'}(V')$ is finite. Claim~\ref{cl:base} implies that the Hausdorff space $X'$ has no isolated points. Then there exist points $u\in U'\setminus I$ and $v\in V'\setminus I$.
Since the space $(X,\Tau)$ is Hausdorff, there exist open sets $O_u,O_v\in\Tau$ such that $u\in O_u$, $v\in O_v$, and $\cl_X(O_u)\cap I=\emptyset=I\cap \cl_X(O_v)$. Since $\Tau'$ is an open refinement of the topology $\Tau$, there exist open sets $U,V\in\Tau$ such that $U\subseteq U'\cap O_u$, $V\subseteq V'\cap O_v$ and the sets $\{u\}\cup U,\{v\}\cup V$ belong to the topology $\Tau'$. The definition of the topology $\Tau'$ implies that the topological space $(X,\Tau')$ has no isolated points and hence the open sets $U,V$ are not empty. Find $k,m\in\w$ such that $U_k\subseteq U$ and $U_m\subseteq V$. Then for the point $x\defeq\lambda_{2^k(2m+1)}$ we have $x\in \overline U_k\cap\overline U_m.$
We claim that $x\in \cl_{X'}(U')\cap\cl_{X'}(V')$.

Fix any neighborhood $W\in\Tau'$ of $x$. By the definition of the topology $\Tau'$, there exist a set $F\in\F_x$ such that $S_x[3F\pm 1]\subseteq W$. By the condition (v), $S_x[3F+1]\subseteq U_k\subseteq U\subseteq U'$ and $S_x[3F-1]\subseteq U_m\subseteq V\subseteq V'$, which implies that $W\cap U'\ne\emptyset\ne W\cap V'$.

Therefore, $x\in\cl_{X'}(U')\cap\cl_{X'}(V')=I$. On the other hand,
$$x=\lambda_{2^k(2m+1)}\in\cl_{X}(U_k)\cap\cl_{X}(U_m)\subseteq\cl_{X}(O_u)\cap\cl_{X}(O_v)\subseteq X\setminus I.$$
This contradiction completes the proof of the Brown property of the topology $\Tau'$.
\end{proof}

\begin{claim}\label{cl:anticompact} The space $X'$ is anticompact.
\end{claim}

\begin{proof} In the opposite case, $X'$ contains an infinite compact subset $K$. Since $\Tau\subseteq\Tau'$, the  set $K$ remains compact in the second-countable Hausdorff space $X$ and being infinite, contains a sequence $\{x_n\}_{n\in\w}\subseteq K$ of pairwise distinct points converging to some point $x\in K$. The compactness of $K\subseteq X'$ and the Hausdorff property of the space $X$ ensure that the identity inclsuion $K\to X$ is a topological embedding, so the sequence $(x_n)_{n\in\w}$ converges to $x$ in the space $X'$.

Now consider the function $r_x:X\to\w$ such that $r_x^{-1}(0)= X\setminus S_x[\IN\setminus\{1\}]$ and $r_x^{-1}(m)=S_x[\{3m-1,3m,3m+1\}]$ for all $m\in\IN$. If for some $m\in\w$ the set $J_m\defeq\{n\in\w:r_x(x_n)=m\}$ is infinite, then choose any set $F\in\F_x$ with $m\notin F$.

If for every $m\in\w$ the set $J_m$ is finite, then the set $I\defeq\{r_x(x_n)\}_{n\in\w}$ is infinite and hence it can be written as the union $I=I_0\cup I_1$ of two disjoint infinite subsets. Since $\F_x$ is an ultrafilter, there exists $k\in\{0,1\}$ such that $I_k\notin \F_x$ and hence the set $F\defeq\IN\setminus I_k$ belongs to $\F_x$. 

In the both cases we can find a set $F\in\F_x$ such that the set $J\defeq\{n\in\w:r_x(x_n)\notin F\}$ is infinite. Then $U_x\defeq\{x\}\cup S_x[3F\pm1]$ is a $\Tau'$-open neighborhood of $x$ such that the set $\{n\in\w:x_n\notin U_x\}$ is infinite, which contradicts the convergence of the sequence $(x_n)_{n\in\w}$ to $x$.
\end{proof}

\begin{claim}\label{cl:self} Every continuous map $f:C\to X'$  from any vertex-connected closed subspace $C$ of $X'=(X,\Tau')$ is either constant or else is the identity embedding.
\end{claim}

\begin{proof} Assume that $f:C\to X'$ is not an identity embedding. Then there exists a point $c\in C$ such that $f(c)\ne c$. By the continuity of $f$, the set $$E\defeq\{x\in C\setminus\{f(c)\}:f(x)=f(c)\}$$ is closed in the space $C\setminus\{f(c)\}$. Since the space $C$ is vertex-connected, its subspace $C\setminus\{f(c)\}$ is connected.
Let us show that the set $E$ is open in $C\setminus \{f(c)\}$. Choose any point $x\in E$ and let $y\defeq f(x)=f(c)\ne x$.

By Claim~\ref{cl:base}, the set $O_y\defeq\{y\}\cup S_y[\IN]$ is an open neighborhood of $y$ in the space $X'$. Consider the function $r_y:X\to \w$  such that $$r_y^{-1}(0)=X\setminus S_y[\IN\setminus\{1\}]\quad\mbox{and}\quad r_y^{-1}(n)=S_y[\{3n-1,3n,3n+1\}]$$ for all $n\in\IN$.

 By the continuity of the function $f$ at $x$ and Claim~\ref{cl:base}, there exist $F\in \F_x$ such that $f[S_x[3F\pm 1]\cap C]\subseteq O_y$.

For $i\in\{-1,1\}$ consider the sets
$$
I_i\defeq\{n\in\IN:S_x(3n+i)\cap C\subseteq f^{-1}(y)\}\quad\mbox{and}\quad J_i\defeq\{n\in\IN:S_x(3n+i)\cap (C\setminus f^{-1}(y))\ne\emptyset\}.
$$
It is clear that $I_i\cup J_i=\IN$.

We claim that $J_i\notin \F_x$. Indeed, assuming that $J_i\in\F_x$ we can choose a function $\varphi:J_i\to C\setminus f^{-1}(y)$ such that $\varphi(n)\in S_x(3n+i)\cap  (C\setminus f^{-1}(y))$ for all $n\in J_i$. Then $$f\circ \varphi[J_i\cap F]\subseteq f[S_x[3F+i]\cap C]\setminus\{y\}\subseteq  O_y\setminus\{y\}=S_y[\IN]$$ and hence $r_y\circ f\circ \varphi[J_i\cap F]\subseteq\IN$. Since the ultrafilters $\F_x$ and $\F_y$ are Rudin--Keisler incomparable, there exist sets $F_x\in\F_x$ and $F_y\in\F_y$ such that  $F_y\cap r_y[f[ \varphi[F_x]]]=\emptyset$. Since $F,J_i\in\F_x$, we can additionally assume that $F_x\subseteq J_i\cap F$. Consider the neighborhood $O'_y\defeq \{y\}\cup S_y[3F_y\pm1]$ of $y$ in $X'$. By the continuity of $f$ at $x$ and Claim~\ref{cl:base}, there exists a set $F'_x\in\F_x$ such that $f[S_x[3F'_x\pm1]]\subseteq O'_y$ and $F'_x\subseteq F_x\subseteq J_i\cap F$. Choose any number $n\in F'_x$ and observe that $r_y\circ f\circ \varphi(n)\in r_y\circ f\circ \varphi[F_x]\subseteq \IN\setminus F_y$ and hence $f\circ\varphi(n)\notin \{y\}\cup S_y[3F_y\pm1]=O'_y$. On the other hand, $f\circ\varphi(n)\in f[S_x[3F'_x\pm1]]\subseteq O'_y$. This contradiction shows that $J_i\notin\F_x$ and hence $I_i=\IN\setminus J_i\in\F_x$.

Then for the set $I\defeq I_{-1}\cap I_1\in\F_x$ and the neighborhood $O_x\defeq(\{x\}\cup S_x[3I\pm1])\setminus\{y\}\in\Tau'$ of $x$ we have $O_x\cap C\subseteq f^{-1}(y)$ and hence $O_x\cap C\subseteq E$, which means that the closed set $E$ is open in $C\setminus\{y\}$. By the connectedness of $C\setminus\{y\}$ the clopen set $E\ni c$ coincides with $C\setminus\{y\}$. Therefore, $f[C\setminus\{y\}]=\{y\}$ and by the continuity of $f$ and density of $C\setminus\{y\}$ in $C$, we have $f[C]=\{y\}$, which means that the function $f$ is constant.
\end{proof}

\begin{claim} The space $X'=(X,\Tau')$ is strongly rigid.
\end{claim}

\begin{proof} By Claim~\ref{cl:Brown}, the space $X'$ is Brown. By Lemma~\ref{l:vertex}, the Brown space $X'$ is vertex-connected and by Claim~\ref{cl:self}, every continuous map $f:X'\to X'$ is either constant or the identity, which means that the space $X'$ is strongly rigid.
\end{proof}

\section{Proof of Theorem~\ref{t:main2}}

We start like in the proof of Theorem~\ref{t:main}: fix any countable second-countable Brown Hausdorff space $X$ and let $\{U_n\}_{n\in\w}\subseteq\Tau_X\setminus\{\emptyset\}$ be a countable base of the topology $\Tau_X$ of $X$. For every $n\in\IN$ choose a point $\lambda_n\in \overline U_k\cap\overline U_m\setminus\{\lambda_i:i<n\}$ where $n=2^k(2m+1)$. For every $x\in X$ and $n\in\IN$, let $O_n(x)\defeq\bigcap\{U_k:k\le n,\;x\in U_k\}$. It is easy to see that $\{O_n(x)\}_{n\in\IN}$ is a neighborhood base of the topology $\Tau_X$ at $x$.

For every $x\in X$, choose a function $s_x:\IN\to X\setminus\{x\}$ such that the following conditions are satisfied:
\begin{itemize}
\item[(i)] $s_x(n)\in O_n(x)\setminus(\{x\}\cup\{s_x(i):1\le i<n\})$ for every $n\in\IN$;
\item[(ii)] if $x=\lambda_{2^k(2m+1)}$ for some $k,m\in\w$, then $s_x[3\IN+1]\subseteq U_k$ and $s_x[3\IN-1]\subseteq U_m$.
\end{itemize}

By Lemma~\ref{l:H} and the Hausdorff property of $X$, for every $x\in X$ there exists a function $S_x:\IN\to \Tau_X$ such that
\begin{itemize}
\item[(iii)] $s_x(n)\in S_x(n)\subseteq O_n(x)$ and $x\notin\cl_X(S_x(n))$ for every $n\in\IN$;
\item[(iv)] $S_x(n)\cap S_x(m)=\emptyset$ for every distinct numbers $n,m\in\IN$;
\item[(v)] if $x=\lambda_{2^k(2m+1)}$ for some $k,m\in\w$, then $\bigcup\limits_{n\in\IN}S_x(3n+1)\subseteq U_k$ and $\bigcup\limits_{n\in\IN}S_x(3n-1)\subseteq U_m$.
\end{itemize}

For any $x\in X$ and a subset $F\subseteq \IN$, let $S_x[F]\defeq\bigcup_{n\in F}S_x(n)$.

Fix any point $a\in X$ and take any set $Y$ such that $Y=X\cup\{b\}$ for some point $b\notin X$. 

Let $r:Y\to X$ be the unique function such that $r(x)=x$ for all $x\in X$ and $r(b)=a$. 

 By Lemma~\ref{l:Kunen}, there exists a family of pairwise Rudin--Keisler incomparable free ultrafilters $(\F_y)_{y\in Y}$ on $\IN$. Let $\Fr\defeq\{F\subseteq\IN:|\IN\setminus F|<\w\}$ be the Fr\'echet filter on $\IN$.

Let $\Tau_Y$ be the topology on $Y$ consisting of the sets $W\subseteq Y$ satisfying the following conditions:
\begin{itemize}
\item for every $y\in W\setminus\{a,b\}$ there exists a set $F\in \F_y$ such that $S_y[3F\pm1]\subseteq W$;
\item if $a\in W$, then there exist sets $F\in \F_a$ and $L\in\Fr$  such that  $S_a[3L]\cup S_a[3F\pm1]\subseteq W$;
\item if $b\in W$, then there exists a set $F\in \F_b$ such that  $S_a[3F]\cup S_a[3F\pm1]\subseteq W$.
\end{itemize}
From now on we consider $Y$ as a topological space endowed with the topology $\Tau_Y$. 

By analogy with Claim~\ref{cl:base} we can prove the following three claims describing neighborhood bases at points of the topological space $Y$.

\begin{claim}\label{cl:base2} For every $y\in Y\setminus\{a,b\}$, the family
$$\mathcal B_y\defeq\big\{\{y\}\cup S_y[3F\pm1]:F\in\F_y\big\}$$is a neighborhood base of the topology $\Tau_Y$ at $y$.
\end{claim}

\begin{claim}\label{cl:base3} The family
$$\mathcal B_a\defeq\big\{\{a\}\cup S_a[3L]\cup S_a[3F\pm1]:L\in\Fr,\;F\in\F_a\big\}$$is a neighborhood base of the topology $\Tau_Y$ at $a$.
\end{claim}

\begin{claim}\label{cl:base4} The family
$$\mathcal B_b\defeq\big\{\{b\}\cup S_a[3F\cup(3F\pm1)]:F\in\F_b\big\}$$is a neighborhood base of the topology $\Tau_Y$ at $b$.
\end{claim}

By Claims~\ref{cl:base2}--\ref{cl:base4}, the Hausdorff property of the space $X$ implies the Hausdorff properties of the subspaces $Y\setminus\{a\}$ and $Y\setminus\{b\}$ of the topological space $Y$.

Now we prove that the topological space $Y$ has all the properties, required in Theorem~\ref{t:main2}: it is strongly rigid, Brown, $k_2$-metrizable, semi-Hausdorff, and contains a non-closed compact metrizable subset. Those properties of $Y$ are established in the following claims.

\begin{claim} The subset $K\defeq\{a\}\cup s_a[3\IN]$ of the space $Y$ is compact and metrizable but not closed in $Y$.
\end{claim}

\begin{proof} By Claim~\ref{cl:base3}, the sequence $\big(s_a(3n)\big)_{n\in\IN}$ converges to the point $a$ in the  topological space $Y$, which implies that the set $K=\{a\}\cup s_a[3\IN]$ is compact. Since $K$ is contained in the Hausdorff subspace $(X,\Tau_X')$ of $(Y,\Tau_Y)$, the compact space $K$ is Hausdorff and being countable, is metrizable (in fact, $K$ is homeomorphic to the compact subspace $\{0\}\cup\{\frac1n:n\in\IN\}$ of the real line). By Claim~\ref{cl:base4}, the point $b\notin K$ belongs to the closure of $K$ in $Y$, witnessing that the compact set $K$ is not closed in $Y$.
\end{proof}

\begin{claim} The space $Y$ is semi-Hausdorff and hence a $T_1$-space.
\end{claim}

\begin{proof} Given two distinct points $x,y\in Y$, we should find a regular open set $U\subseteq Y$ such that $x\in U$ and $y\notin Y$. The definition of the topology $\Tau_Y$ ensures that the function $r:Y\to X$ is continuous. If $r(x)\ne r(y)$, then by the Hausdorff property of $X$, there exist disjoint open sets $V,W$ in $X$ such that $r(x)\in V$ and $r(y)\in W$. By the continuity of $r$, the preimages $r^{-1}[V]$ and $r^{-1}[W]$ are disjoint open sets in $Y$ such that $x\in r^{-1}[V]$ and $y\in r^{-1}[W]$. Let $U$ be the interior of the closure of the set $r^{-1}[V]$ in $Y$. It follows that $U$ is a regular open neighborhood of $x$ such that $\overline U\subseteq\overline{r^{-1}[V]}\subseteq Y\setminus r^{-1}[W]\subseteq Y\setminus\{y\}$.

It remains to consider the case $r(x)=r(y)$ which happens if and only if $\{x,y\}=\{a,b\}$.  Since the ultrafilters $\F_x$ and $\F_y$ are Rudin--Keisler incomparable, there are sets $F_x\in\F_x$ and $F_y\in\F_y$ such that $F_x\cap F_y=\emptyset$. By Claims~\ref{cl:base3} and \ref{cl:base4}, the sets 
$$U_x\defeq\{x\}\cup S_a[3\IN]\cup S_a[3F_x\pm1]\quad\mbox{and}\quad
U_y\defeq\{y\}\cup S_a[3\IN]\cup S_a[3F_y\pm1]$$are open neighborhoods of the points $x$ and $y$ in the space $Y$, respectively. 

Let $U$ be the interior of the closure $\overline{U_x}$ of the set $U_x$ in $Y$. We claim that $y\notin U$. In the opposite case, by Claim~\ref{cl:base3}, there exists a set $F\in\F_y$ such that
$S_a[3F\pm1]\subseteq \overline{U_x}$ and $F\subseteq F_y$. It follows from $F_x\cap F_y=\emptyset$ that $$U_x\subseteq Y\setminus S_a[3F_y\pm1]$$ and hence $\overline{U_x}\subseteq Y\setminus S_a[3F_y\pm1]$, by the openness of the set $S_a[3F_y\pm1]$ in $X$ and $Y$. Then we obtain a contradiction:
$$S_a[3F\pm1]\subseteq\overline{U_x}\subseteq Y\setminus S_a[3F_y\pm1]\subseteq Y\setminus S_a[3F\pm1],$$
which shows that $y\notin U$, witnessing that the space $Y$ is semi-Hausdorff.
\end{proof}

By analogy with Claim~\ref{cl:anticompact} one can prove

\begin{claim}\label{cl:anticompact2} The Hausdorff subspace $Y\setminus\{a\}$ of the space $Y$ is anticompact.
\end{claim}

\begin{claim}\label{cl:k-metrizable} The space $Y$ is $k_2$-metrizable.
\end{claim}

\begin{proof} On the space $Y$ consider the metric $d:Y\times Y\to[0,1]$ defined by 
$$
d(x,y)\defeq\begin{cases}0&\mbox{if $x=y$};\\
\tfrac1n+\tfrac1m&\mbox{if $S_a(3n)\ni x\ne y\in S_a(3m)$ for some $n,m\in\IN$};\\
\tfrac1m&\mbox{if $x=a$ and $y\in S_a(3m)$ for some $m\in\IN$};\\
\tfrac1n&\mbox{if $y=a$ and $x\in S_a(3n)$ for some $n\in\IN$};\\
1&\mbox{otherwise}.
\end{cases}
$$
Observe that the point $a$ is a unique non-isolated point of the metric space $M\defeq (Y,d)$. The definition of the topology of the space $Y$ at $a$ ensures that the identity map $i:M\to Y$ is continuous. 

To prove that the map $i$ is proper, take any infinite compact subset $K\subseteq Y$. The anticompactness of the space $Y\setminus\{a\}$ implies that $a\in K$ and hence $r[K]=K\setminus\{b\}$. Claims~\ref{cl:base3} and \ref{cl:base4} imply that the map $r:Y\to Y\setminus\{a\}$ is continuous and hence the subspace $K\setminus\{b\}=r[K]$ of the Hausdorff space $Y\setminus\{b\}$ is compact. We claim that the set $K\setminus\{b\}$ is compact in the metric space $M$. In the opposite case we can find a sequence $(x_n)_{n\in\w}$ of pairwise distinct points of $K\setminus\{a,b\}$ that has no convergent subsequence in the metric space $M$.
The continuity of the identity map $Y\setminus\{b\}\to X$ and the second-countability of the Hausdorff space $X$ imply that the compact subspace $K\setminus\{b\}$ of the Hausdorff space $Y\setminus\{b\}$ is second-countable and hence metrizable. Then the sequence  $(x_n)_{n\in\omega}$ has a subsequence $(x_{n_k})_{k\in\omega}$ that converges to some point $x_\infty$ of the compact metrizable space $K\setminus\{b\}$. Replacing the sequence $(x_n)_{n\in\w}$ by the convergent subsequence $(x_{n_k})_{k\in\w}$, we lose no generality assuming that the sequence $(x_n)_{n\in\omega}$ converges to $x_\infty$.  The anticompactness of the space $Y\setminus\{a\}$ implies that $x_\infty=a$ and hence the sequence $(x_n)_{n\in\w}$ converges to $a$ in the compact metrizable space $K\setminus\{b\}$. Repeating the argument from the proof of Claim~\ref{cl:anticompact}, we can show that the set $\{n\in\w:x_n\notin S_a[3\IN]\}$ is finite.  For every $m\in\IN$ the condition (iii) of the choice of the set $S_a(3m)$ ensures that $a\notin\cl_X(S_a(3m))=\emptyset$ and hence the set $\{n\in\w:x_n\in S_a(3m)\}$ is finite (by the convergence of the sequence $(x_n)_{n\in\w}$ to $a$. Now the definition of the metric $d$ ensures that the sequence $(x_n)_{n\in\w}$ converges to $a$ in the metric space $M$, which contradicts the choice of $(x_n)_{n\in\w}$. This contradiction finishes the proof of the compactess of the set $K\setminus\{b\}$ in the metric space $M$. The compactness of $K\setminus\{b\}$ implies the compactness of the set $K$ in the metric space $M$, and the properness of the identity map $i:M\to Y$.

To prove that the space $Y$ is $k_2$-metrizable, take any continuous map $f:H\to Y$ defined on a compact Hausdorff space $H$. We need to prove that the map $i^{-1}\circ f:H\to M$ is continuous. The compactness of $H$ and the continuity of $f$ imply that the image $K\defeq f[H]$ is a compact subset of the $T_1$-space $Y$.
Consider the open subspace $H_b\defeq H\setminus f^{-1}(b)$ of the compact Hausdorff space $H$.

By the properness of identity map $i:M\to Y$, the set $K$ remains compact with respect to the metric $d$. Since $b$ is an isolated point of the metric space $M$, then set $K\setminus\{b\}$ is compact in $M$ and also in the topological space $Y$. Since the space $Y\setminus\{b\}$ is Hausdorff, the restriction $i{\restriction}_{K\setminus\{b\}}:K\setminus\{b\}\to Y\setminus\{b\}$ is a topological embedding and hence the restriction $i^{-1}\circ f{\restriction}_{H_b}:H_b\to M$ is continuous, being the composition of the continuous maps $f{\restriction}_{H_b}$ and $i^{-1}{\restriction}_{K\setminus\{b\}}$. It remains to prove that the set $H_b$ is closed in  $H$.

If $b\notin K$, then $H_b=H$ is closed in $H$. If $K$ is finite, then $H_b=\bigcup_{y\in K\setminus \{b\}}f^{-1}(y)$ is closed in $H$ as a finite union of preimages of points (which are closed by the continuity of $f$ and the $T_1$ separation property of $Y$).

So, assume that the compact set $K$ is infinite and  $b\in K$. By the definition of the metric $d$, the infinite compact set $K$ of the metric space $M$ has the point  $a$ as the unique non-isolated point. Then $K$ is countable and hence $K\setminus\{a\}=\{x_n\}_{n\in\w}$ for some sequence of pairwise distinct points that converge to $a$, starting with $x_0=b$.

By the $T_1$ separation property of the space $Y$, for every $y\in Y$ the preimage $f^{-1}(y)$ is a closed subset of the compact Hausdorff space $H$. By the normality of the compact Hausdorff space $H$, there exist two disjoint open sets $U,V\subseteq H$ such that $f^{-1}(a)\subseteq U$ and $f^{-1}(b)\subseteq V$. We claim that the set $\Omega\defeq\{n\in\IN:f^{-1}(x_n)\cap V\ne\emptyset\}$ is finite. In the opposite case, we can repeat the argument of the proof of Claim~\ref{cl:anticompact} and find an open neighborhood $O_b\subseteq Y$ of $b$ such that the set $\Lambda=\{n\in\Omega:x_n\notin O_b\}$ is infinite.  For every $n\in\Lambda$, choose a point $z_n\in V\cap f^{-1}(x_n)$. By the compactness of $H$, there exists a point $c\in H$ such that every neighborhood $O_c\subseteq H$ of $c$ contains infinitely many points $z_n$, $n\in\Lambda$. Since the open neighborhood $f^{-1}[O_b]$ of the set $f^{-1}(b)$ contains no point $z_n$ with $n\in\Lambda$, the point $c$ cannot belong to the set $f^{-1}(b)=f^{-1}(x_0)$.
Since $\{z_n\}_{n\in\Lambda}\subseteq V\subseteq H\setminus U$, the point $c$ belongs to the closed set $H\setminus U$ and hence $c\notin f^{-1}(a)$. Then $c\in f^{-1}(x_k)$ for some $k\in\IN$. Using the convergence of the sequence $(x_n)_{n\in\IN}$ to $a$ in the Hausdorff space $Y\setminus\{b\}$, we can find an open neighborhood $W\subseteq Y$ of $x_k$ such that $\{n\in\IN:x_n\in W\}=\{k\}$. Then $f^{-1}[W]$ is a neighborhood of the point $c$ containing no point $z_n$ with $n\in\Lambda\setminus\{k\}$, which contradicts the choice of the point $c$. This contradiction shows that the set $\Omega$ is finite. Then the set $H_b=H\setminus f^{-1}(b)=(H\setminus V)\cup\bigcup_{n\in\Omega}f^{-1}(x_n)$ is closed in $H$ and its complement $f^{-1}(b)=H\setminus H_b$ is open in $H$. The restriction $i^{-1}\circ f{\restriction}_{f^{-1}(b)}:f^{-1}(b)\to \{b\}\subseteq M$ is constant and hence continuous. The continuity of the restrictions $i^{-1}\circ f{\restriction}_{f^{-1}(b)}$ and $i^{-1}\circ f{\restriction}_{H_b}$ and the openness of the sets $f^{-1}(b)$ and $H_b=H\setminus f^{-1}(b)$ in $H$ imply the continuity of the map $i^{-1}\circ f:H\to M$ and finally, the $k_2$-metrizability of the space $Y$. 
\end{proof}

\begin{claim}\label{cl:Brown2} The space $Y$ is Brown.
\end{claim}

\begin{proof} Assuming that the space $Y$ is not Brown, we can find two nonempty open sets $U,V\in\Tau_Y$ such that the intersection $I\defeq\cl_Y(U)\cap\cl_Y(V)$ is finite. Then the set $J\defeq I\cup\{a,b\}$ is finite as well. Claims~\ref{cl:base2}--\ref{cl:base4} imply that the semi-Hausdorff space $Y$ has no isolated points. Then there exist points $u\in U\setminus J\subseteq X$ and $v\in V\setminus J\subseteq X$.
Since the space $X$ is Hausdorff, there exist open sets $O_u,O_v\subseteq X$ such that $u\in O_u$, $v\in O_v$, and $\cl_X(O_u)\cap J=\emptyset=J\cap \cl_X(O_v)$. By Claim~\ref{cl:base2} and the property (iii) of the functions $S_u$ and $S_v$, there exist  sets $F_u\in \F_u$ and $F_v\in F_v$ such that $S_u[3F_u\pm1]\subseteq O_u\cap U$ and $S_v[3F_v\pm1]\subseteq O_v\cap V$. Find $k,m\in\w$ such that $U_k\subseteq S_u[3F_u\pm1]$ and $U_m\subseteq S_v[3F_v\pm1]$. Then for the point $x\defeq\lambda_{2^k(2m+1)}$ we have $x\in \overline U_k\cap\overline U_m.$
We claim that $x\in \cl_{Y}(U)\cap\cl_{Y}(V)$.

Fix any neighborhood $W\in\Tau_Y$ of $x$. By Claims~\ref{cl:base2} and \ref{cl:base3} and the property (v) of the function $S_x$, there exists a set $F\in\F_x$ such that $$S_x[3F\pm 1]\subseteq W,\;\;S_x[3F+1]\subseteq U_k\subseteq S_u[3F_u\pm1]\subseteq U,\mbox{ \  and \ }S_x[3F-1]\subseteq U_m\subseteq S_v[3F_v\pm 1]\subseteq V,$$ which implies that $W\cap U\ne\emptyset\ne W\cap V$.

Therefore, $x\in\cl_{Y}(U)\cap\cl_{Y}(V)=I\subseteq J$. On the other hand,
$$x\in\cl_{X}(U_k)\cap\cl_{X}(U_m)\subseteq\cl_{X}(O_u)\cap\cl_{X}(O_v)\subseteq X\setminus J.$$
This contradiction completes the proof of the Brown property of the topological space $Y$.
\end{proof}

\begin{claim}\label{cl:self2} The space $Y$ is strongly rigid.
\end{claim}

\begin{proof} Take any continuous map $f:Y\to Y$.  Separately we shall consider two cases.
\smallskip

1. First assume that there exists a point $c\in Y\setminus\{a,b\}$ such that $c\ne f(c)\notin\{a,b\}$. Since $Y$ is a Brown space, its subspace $Z\defeq Y\setminus\{a,b,f(c)\}$ is connected, by Lemma~\ref{l:vertex}. Since the space $Y$ is semi-Hausdorff and hence a $T_1$-space, the set $$E\defeq\{z\in Z:f(z)=f(c)\}$$ is closed in the $Z$.  Repeating the argument of the proof of Claim~\ref{cl:self}, we can show that the set $E\ni c$ is open in $Z$ and hence $E=Z$ by the connectedness of $Z$. Then $f[Z]=f[E]=\{f(c)\}$. Since the connected $T_1$-space $Y$ has no isolated points, the set $Z=Y\setminus\{a,b,f(c)\}$ is dense in $Y$ and hence $f[Y]\subseteq\overline{f[Z]}=\overline{\{f(c)\}}=\{f(c)\}$, which means that the function $f$ is constant.
\smallskip

2. Next, assume that for every point $c\in Y\setminus\{a,b\}$ with $c\ne f(c)$ we have $f(c)\in\{a,b\}$. By Lemma~\ref{l:vertex}, the subspace $Z\defeq Y\setminus\{a,b\}$ of $Y$ is connected. Since $Y$ is a $T_1$-space, the set $E\defeq\{z\in Z:f(z)\in\{a,b\}\}$ is closed in $Z$. We claim that $E$ is open in $Z$. Indeed, given any point $x\in E$, we observe that $x\ne f(x)\in\{a,b\}$. By the definition of the topology $\Tau_Y$, there exist disjoint open sets $U,V\subseteq Y$ such that $x\in U$ and $\{a,b\}\subseteq V$. The continuity of $f$ yields a neighborhood $O_x\subseteq U\cap Z$ of $x$ such that $f[O_x]\subseteq V$. Then $O_x\cap f[O_x]=\emptyset$ and hence $f[O_x]\subseteq\{a,b\}$, according to our assumption. Then $O_x\subseteq E$, witnessing that the set $E$ is open in $Z$. By the connectedness of $Z$, we have $E=\emptyset$ or $E=Z$. If $E=Z$, then the connected subspace $f[Z]\subseteq \{a,b\}$ of the $T_1$-space is a singleton. The density of $Z$ in $Y$ implies that $f[Y]$ is a singleton as well and hence $f$ is constant. If $E=\emptyset$, then for every $z\in Z$ we have $f(z)\notin \{a,b\}$ and hence $f(z)=z$ by our assumption. We claim that $f:Y\to Y$ is the identity map of $Y$. In the opposite case, we can find a point $x\in\{a,b\}$ such that $f(x)\ne x$. 

If $f(x)\notin\{a,b\}$, then by the definition of the topology $\Tau_Y$, there exist two disjoint open sets $U,V\subseteq Y$ such that $f(x)\in U$ and $\{a,b\}\subseteq V$. By the continuity of $f$ at $x\in\{a,b\}$, there exists an open neighborhood $O_x\subseteq V$ of $x$ such that $f[O_x]\subseteq U$. Since the connected $T_1$-space $Y$ has no isolated points, the open subset $O_x$ of $Y$ contains a point $c\in O_x\setminus\{a,b\}$. It follows that $f(c)\in f[O_x]\subseteq U\subseteq Y\setminus\{a,b\}$. Now our assumption guarantees that $f(c)=c\in V$ and hence $f(c)\in U\cap V=\emptyset$. This contradiction shows that $f(x)\in\{a,b\}$ and hence $\{x,f(x)\}=\{a,b\}$. Consider the point $y\defeq f(x)\in\{a,b\}\setminus\{x\}$.  Since the ultrafilters $\F_x$ and $\F_{y}$ are distinct, there exists a set $F_y\in\F_y\setminus\F_x$. By Claims~\ref{cl:base3} and  \ref{cl:base4}, the set $U_y\defeq\{y\}\cup S_a[3\IN]\cup S_a[3F_y\pm1]$ is a neighborhood of $y$ in $Y$. By the continuity of the map $f$ at $x$ there exists a neighborhood $U_x$ of $x$ in $Y$ such that $f[U_x]\subseteq U_y$. By Claims~\ref{cl:base3} and \ref{cl:base4}, there exists a set $F_x\in\F_x$ such that $S_a[3F_x\pm 1]\subseteq U_x$. Since $F_y\notin\F_x$, we can replace $F_x$ by a smaller set in $\F_x$ and additionally assume that $F_x\subseteq \IN\setminus F_y$. Now take any number $n\in F_x$ and consider the element $z\defeq s_a(3n+1)\in U_x\setminus\{a,b\}$. Our assumption implies that $s_a(3n+1)=z=f(z)\in f[U_x]\subseteq U_y=\{y\}\cup S_a[3\IN]\cup S_a[3F_y\pm 1]$ and hence $n\in F_y$, which is impossible as $F_x\cap F_y=\emptyset$. This contradiction shows that $f$ is the identity map of $Y$. 
\end{proof}

\section{The interplay between various $k$-properties of topological spaces}\label{s:US}

In this section we prove nontrivial implications in the diagram, drawn before Theorem~\ref{t:main2}.

\begin{proposition}\label{p:k2=>k1} Every $k_2$-metrizable space $X$ is $k$-metrizable and $T_1$.
\end{proposition}

\begin{proof} Being $k_2$-metrizable, the space $X$ admits a continuous proper bijective map $p:M\to X$ from a metrizable space $M$ such that for any continuous map $f:H\to X$ from a compact Hausdorff space $H$ the map $p^{-1}\circ f:H\to M$ is continuous. The map $p$ witnesses that the space $X$ is $k$-metrizable. Assuming that $X$ is not a $T_1$-space, we can find a point $x\in X$ such that the singleton $\{x\}$ is not closed in $X$. Let $y\ne x$ be any point in the closure of $\{x\}$ in $X$. Consider the compact metrizable subspace $H=\{0\}\cup\{\frac1n:n\in\IN\}$ of the real line and the function
$f:H\to \{x,y\}\subseteq X$ defined by $f(0)=y$ and $f(\frac1n)=x$ for all $n\in\IN$. It follows from $y\in\overline{\{x\}}$ that the function $f$ is continuous. On the other hand, the function $p^{-1}\circ f:H\to \{x,y\}\subseteq M$ is discontinuous, which contradicts the choice of the map $p$. This contradiction shows that $X$ is a $T_1$-space.
\end{proof}

\begin{proposition}\label{p:k2d=>kd} A topological space $X$ is $k_2$-discrete if and only if $X$ is a $k$-discrete $T_1$-space.
\end{proposition}

\begin{proof} The ``only if'' part follows from Proposition~\ref{p:k2=>k1}. To prove the ``if'' part, assume that $X$ is a $k$-discrete $T_1$-space. Then there exists a bijective proper map $p:D\to X$ from a discrete space $D$. We claim that the map $p$ witnesses that the space $X$ is $k_2$-discrete. Indeed, take any continuous map $f:H\to X$ from a compact Hausdorff space. By the continuity of $f$, the image $K=f[H]$ is a compact subset of $X$. The properness of the map $p$ ensures that $p^{-1}[K]$ is a compact subset of the discrete space $D$. Then $K$ is finite and by the $T_1$ separation property of $X$, $K$ is a finite discrete subspace of $X$. Then the map $i^{-1}\circ f:H\to D$ is continuous, being the composition of the continuous maps $f:H\to X$ and $i^{-1}{\restriction}_K:K\to D$. This shows that the space $X$ is $k_2$-discrete.
\end{proof}

Propositions~\ref{p:k2=>k1} and \ref{p:k2d=>kd} motivate the following question.

\begin{question} Is every $k$-metrizable $T_1$-space $k_2$-metrizable?
\end{question}

\begin{proposition} Every $k_2$-metrizable space $X$ is $k_\Delta$-Hausdorff.
\end{proposition}

\begin{proof} Assume that $X$ is $k_2$-metrizable. To show that $X$ is $k_\Delta$-Hausdorff, take any continuous map $f:H\to X\times X$ from a compact Hausdorff space $H$. Let $\pi_1,\pi_2:X\times X\to X$ be the coordinate projections and $f_i=\pi_i\circ f:H\to X$ for every $i\in\{1,2\}$. By the $k_2$-metrizability of $X$, there exists a bijective continuous proper map $p:M\to X$ from a metrizable space $M$ such that the maps $p^{-1}\circ f_1$ and $p^{-1}\circ f_2$ are continuous. By the Hausdorff property of $M$, the set $f^{-1}[\Delta_X]=\{x\in H:f_1(x)= f_2(x)\}=\{x\in H:p^{-1}\circ f_1(x)=p^{-1}\circ f_2(x)\}$ is closed, witnessing that the space $X$ is $k_\Delta$-Hausdorff.
\end{proof}

\begin{proposition} Every $k_\Delta$-Hausdorff space $X$ is a US-space.
\end{proposition}
 
\begin{proof}  Assuming that $X$ is not a $US$-space, we can find a sequence $(x_n)_{n\in\IN}$ in $X$ that converges to two distinct points $a,b\in X$. Consider the compact Hausdorff space $H=\{0\}\cup\{\frac1n:n\in\IN\}\subseteq\IR$ and the continuous function $f:H\to X\times X$ such that $f(0)=(a,b)$ and $f(\frac1n)=(x_n,x_n)$ for all $n\in\IN$. Since the preimage $f^{-1}[\Delta_X]=\{\frac1n:n\in\IN\}$ is not closed in $H$, the space $X$ is not $k_\Delta$-Hausdorff.
\end{proof}

\begin{proposition}\label{p:k=>US} Every $k$-metrizable $T_1$-space $X$ is a US-space.
\end{proposition}
 
\begin{proof}  By the $k$-metrizability of $X$, there exists a proper continuous bijective map $p:M\to X$, defined on a metrizable space $M$. Assuming that $X$ is not a $US$-space, we can find a sequence $(x_n)_{n\in\IN}$ in $X$ that converges to two distinct points $a,b\in X$. Since $X$ is a $T_1$-space, the set $\{x_n\}_{n\in\IN}$ is infinite. Replacing $(x_n)_{n\in\w}$ by a suitable subsequence, we can assume that $\{a,b\}\cap\{x_n\}_{n\in\w}=\emptyset$ and $x_n\ne x_m$ for any distinct numbers $n,m\in\w$. Since the sequence $(x_n)_{n\in\omega}$ converges to $a$, the set $K_a=\{a\}\cup\{x_n\}_{n\in\w}\subseteq X$ is compact and  its preimage $p^{-1}[K_a]$ is an infinite compact subset of the metrizable space $M$. Let $c$ be a non-isolated point of the compact infinite set $p^{-1}[K]$. Consider the infinite set $\Omega\in\{n\in\w:x_n\ne p(c)\}$ and the compact subset $K_b=\{b\}\cup\{x_n\}_{n\in\Omega}$ of $X$. By the properness, the set $p^{-1}[K_b]$ is compact and then $F=p^{-1}[K_b]\cap p^{-1}[K_a]=\{p^{-1}(x_n)\}_{n\in\Omega}$ is compact, too, which is not true  because $c\in \overline F\setminus F$. This contradiction shows that $X$ is  a $US$-space.
\end{proof}

The following simple example shows that the $T_1$-restriction cannot be removed from Proposition~\ref{p:k=>US}.

\begin{example} The space $X=\{0,1\}$ with the topology $\{\emptyset,\{0\},\{0,1\}\}$ is $k$-metrizable but not a US-space because the constant sequence $(x_n)_{n\in\w}\in\{0\}^\w$ converges to $0$ and to $1$ simultaneously.
\end{example} 
\newpage



\end{document}